\documentclass[11pt]{article}
\usepackage[margin=1.05in]{geometry}
\usepackage{amsmath,amssymb,amsthm,bm,graphicx,booktabs,multirow}
\usepackage[hidelinks]{hyperref}

\newtheorem{theorem}{Theorem}[section]
\newtheorem{lemma}[theorem]{Lemma}

\theoremstyle{definition}
\newtheorem{definition}[theorem]{Definition}
\newtheorem{example}[theorem]{Example}
\theoremstyle{remark}
\newtheorem{remark}[theorem]{Remark}

\newcommand{\R}{\mathbb{R}}
\newcommand{\PR}{P}

\newcommand{\tr}{\mathrm{tr}}
\newcommand{\diag}{\mathrm{diag}}
\newcommand{\1}{\mathbf{1}}

\title{\bf Pitman closest equivariant estimators under multivariate\\ scale and location--scale models}
\author{Yihong Liu\ and\ Haojin Zhou
\thanks{Corresponding author.
Email: \href{mailto:haojin_zhou@hotmail.com}{haojin\_zhou@hotmail.com}.
ORCiD: 0000-0003-0802-099X.}\\
{\normalsize Pediatric Research Institute}\\
{\normalsize Guangzhou Women and Children's Medical Center}\\
{\normalsize Guangzhou Medical University}\\
{\normalsize Guangzhou, Guangdong, P.R. China}}
\date{}

\begin{document}
\maketitle

\begin{abstract}
\noindent
For multivariate scale and location--scale models with independent components, we
extend the univariate results of Zhou and Nayak (2012) and derive optimum equivariant
estimators under the generalized Pitman closeness criterion. We first show, by a
counterexample, that in the multivariate case the Pitman closeness comparison within the
class of equivariant estimators is not transitive, so that a Pitman closest equivariant
estimator does not exist in general. We then enlarge the transformation group by the
coordinate permutations---equivalently, impose the formal equivariance principle of
Berger (1985) across isomorphic component problems, in the spirit of the separable rules
of Robbins' (1951) compound decision theory---and show that within the resulting
restricted class an optimum is restored. A multivariate median lemma based on
strictly convex losses then yields explicit Pitman closest equivariant estimators of the
scale parameters, powers of the scale parameters, and the location parameters, given by
median-adjusted versions of any given equivariant estimator. Applications to the
multivariate uniform and multivariate normal distributions are worked out in detail.
Monte Carlo experiments for the Rayleigh distribution and for a competing risks model
with Rayleigh component lifetimes confirm that the proposed estimators dominate the
maximum likelihood and Bayes estimators under the Pitman closeness criterion, and a real
industrial data set on ball bearing failure times illustrates the feasibility of the
method in practice.


\end{abstract}

\section{Introduction}

Comparing estimators by the probability that one is closer to the target than the other
goes back to Pitman (1937). Peddada (1985) and Rao et al.\ (1986) extended Pitman's
original definition, which was based on absolute error, to general loss functions, and
the resulting criterion is now commonly known as the (generalized) Pitman closeness
criterion (PCC). Formally, for two estimators $T_1(X)$ and $T_2(X)$ of a parameter
$\theta$ (or a function of it) under a loss function $L(d,\theta)$, the Pitman closeness
of $T_1$ relative to $T_2$ is
\begin{equation}\label{eq:pc}
\mathrm{PC}(T_1,T_2;\theta)=\PR_\theta\!\big[L\big(T_1(X),\theta\big)<L\big(T_2(X),\theta\big)\big],
\end{equation}
and $T_1$ is said to be Pitman closer than $T_2$ if
$\mathrm{PC}(T_1,T_2;\theta)\ge \mathrm{PC}(T_2,T_1;\theta)$ for all $\theta$, with strict
inequality for some $\theta$. We refer to Keating et al.\ (1993) for an extensive
discussion.

Pitman (1937) recognized that the criterion is not transitive, but also showed that it
can be used to select an estimator within certain classes. Nayak (1990) formalized this
idea through equivariance: he showed that if a decision problem is invariant under a
group $G$ and $T_1,T_2$ are equivariant, then $\mathrm{PC}(T_1,T_2;\theta)$ is constant
on the orbits of the parameter space, and hence independent of $\theta$ when the
parameter space is transitive. For univariate location, scale and location--scale models,
this phenomenon together with a median argument produces Pitman closest equivariant
estimators; see Ghosh and Sen (1989), Kubokawa (1990, 1991), Kourouklis (1995a, 1995b,
1996) and, for a unified treatment including prediction, Zhou and Nayak (2012). A key
feature of the approach is its robustness with respect to the loss function: the optimum
estimators depend on the loss only through mild monotonicity or convexity conditions.

With multivariate data now ubiquitous, it is natural to ask how these results extend to
multivariate scale and location--scale families. The multivariate case, however, is
genuinely different. We show in Section~2 that the non-transitivity of the PCC, which is
neutralized by equivariance in univariate location--scale problems, reappears in the
multivariate setting: one can construct equivariant estimators $T_1,T_2,T_3$ such that
each is Pitman closer than the previous one, so that no Pitman closest equivariant
estimator exists within the unrestricted equivariant class. This phenomenon parallels the
well-known inadmissibility of best invariant procedures in multivariate problems under
risk-based criteria (Stein, 1956; Portnoy and Stein, 1971), although under the PCC the
difficulty is non-transitivity rather than inadmissibility. Guided by the formal
equivariance principle (Berger, 1985), we enlarge the diagonal transformation group by
the coordinate permutations: when the $p$ component problems are isomorphic, equivariance
under the enlarged group requires the same estimating function in every component. The
resulting restricted class coincides with the separable rules of compound decision theory
(Robbins, 1951), and within it a Pitman closest member exists and can be derived through
a multivariate extension of the median lemma of Zhou and Nayak (2012), based on strictly
convex losses.

The rest of the paper is organized as follows. In the remainder of this section we state
the multivariate median lemma that will be used repeatedly. Section~2 discusses the
multivariate Pitman closeness comparison, the non-transitivity counterexample, and the
formal equivariance restriction. Section~3 treats the multivariate scale family and
Section~4 the multivariate location--scale family; in each case we characterize the
equivariant estimators through a maximal invariant and derive the Pitman closest
equivariant estimators, with worked examples. Section~5 presents Monte Carlo experiments
and a real data analysis for the Rayleigh distribution and for a competing risks model
with Rayleigh component lifetimes. Section~6 contains concluding remarks.

We close this section with the multivariate median lemma. Throughout, for a univariate
random variable $W$, any number $m$ with $\PR(W<m)\le 1/2$ and $\PR(W\le m)\ge 1/2$ is
called a median of $W$.

\begin{lemma}\label{lem:median}
Let $Z$ and $T$ be random vectors in $\R^p$ with $\PR(T=\mathbf 0)=0$, and let
$\theta\in\R^p$ be a constant vector. Suppose $L$ is a strictly convex function on $\R^p$
with $L(\mathbf 0)=0$. Let $m(Z)$ be a function of $Z$ such that, for each value $z$ of
$Z$, $m(z)$ is a median of the conditional distribution given $Z=z$ of the univariate
random variable $c^*$ defined by
\[
T^\intercal\,\nabla L\big(\theta-c^*T\big)=0 .
\]
Then, for any other (univariate) function $k$ of $Z$,
\begin{equation}\label{eq:medianlemma}
\PR\big[L\big(\theta-m(Z)T\big)<L\big(\theta-k(Z)T\big)\big]\ \ge\ \tfrac12\,\PR\big[k(Z)\ne m(Z)\big].
\end{equation}
\end{lemma}

\begin{proof}
Let
\[
q(z)=\PR\big[L\big(\theta-m(Z)T\big)<L\big(\theta-k(Z)T\big)\,\big|\,Z=z\big],
\]
so that the left side of (\ref{eq:medianlemma}) equals $E[q(Z)]$. Fix $z$ and write
$c=m(z)$. Since $L$ is strictly convex, it is twice differentiable almost everywhere
(Evans and Gariepy, 1992, p.~242) with positive semidefinite Hessian, and along the ray
$\{\theta-cT:c\in\R\}$ the function $\phi(c)=L(\theta-cT)$ is strictly convex in $c$
because $\PR(T=\mathbf 0)=0$. Hence $\phi$ attains its minimum at $c^*$ satisfying
$\phi'(c^*)=-T^\intercal\nabla L(\theta-c^*T)=0$, and $\phi$ is strictly decreasing to
the left of $c^*$ and strictly increasing to the right of $c^*$. It follows that, for
$k(z)>m(z)$,
\[
c^*\le m(z)<k(z)\ \Longrightarrow\ L\big(\theta-m(z)T\big)<L\big(\theta-k(z)T\big),
\]
and therefore
\[
q(z)\ \ge\ \PR\big[c^*\le m(z)\,\big|\,Z=z\big]\ \ge\ \tfrac12 ,
\]
since $m(z)$ is a conditional median of $c^*$ given $Z=z$. Similarly, if $k(z)<m(z)$,
considering the event $c^*\ge m(z)$ gives $q(z)\ge \PR[c^*\ge m(z)\mid Z=z]\ge 1/2$.
Since $q(z)=0$ when $k(z)=m(z)$, taking expectation over $Z$ completes the proof.
\end{proof}

Lemma~\ref{lem:median} reduces to Lemma~1.1 of Zhou and Nayak (2012) when $p=1$ and the
loss is $h((\theta-cT)/T)$-type with strictly monotone $h$: in the univariate case the
minimizer condition $T^\intercal\nabla L(\theta-c^*T)=0$ is exactly the crossing point at
which the two losses are tied, so a median of $c^*$ is Pitman closest among all
adjustments of the form $k(Z)T$. The strict convexity of $L$ is essential in the
multivariate case; note that additive losses $L(d,\theta)=\sum_{i=1}^p h(d_i,\theta_i)$
with strictly convex univariate $h$ are strictly convex (Boyd and Vandenberghe, 2004,
p.~79).

\section{Pitman closeness and equivariance in the multivariate case}

\subsection{Equivariance and orbit constancy}

The decision-theoretic framework of equivariance is standard (Berger, 1985; Lehmann and
Casella, 1998). A decision problem $(\mathcal X,\mathcal P,\Theta,\mathcal D,L)$ is
invariant under a group $G$ of transformations of $\mathcal X$ if for each $g\in G$ there
exist transformations $\bar g$ of $\Theta$ and $\tilde g$ of $\mathcal D$ such that the
model and the loss are preserved; a decision rule $\delta$ is equivariant if
$\delta(g(x))=\tilde g(\delta(x))$ for all $g\in G$ and $x\in\mathcal X$. The orbit
constancy result of Nayak (1990) carries over to the multivariate case without change,
and we record it for completeness.

\begin{theorem}\label{thm:orbit}
Suppose a decision problem is invariant under a group $G$ and $T_1,T_2$ are two
equivariant decision rules. Then $\mathrm{PC}(T_1,T_2;\theta)$ is constant on the orbits
of $\Theta$; in particular, if $\Theta$ is transitive under $\bar G$, then
$\mathrm{PC}(T_1,T_2;\theta)$ is independent of $\theta$.
\end{theorem}

In the univariate scale and location--scale problems of Zhou and Nayak (2012),
Theorem~\ref{thm:orbit} reduces every pairwise comparison of equivariant estimators to a
single parameter value, and the median lemma then produces a Pitman closest equivariant
estimator. In the multivariate case this reduction alone is not enough, as we show next.

\subsection{Non-transitivity within the equivariant class}

\begin{example}\label{ex:nontransitive}
Consider estimating $\bm\theta=(\theta_1,\dots,\theta_p)$ in a multivariate scale
problem (formally defined in Section~3) under a componentwise additive loss
\begin{equation}\label{eq:addloss}
L(\bm d,\bm\theta)=\sum_{i=1}^p h(d_i,\theta_i),
\end{equation}
where $h$ is strictly convex and scale invariant. Let $T_u$ denote the vector whose
$i$-th component is the univariate Pitman closest equivariant estimator of $\theta_i$ of
Zhou and Nayak (2012), computed from the $i$-th component sample alone, and let $T_m$
denote the estimator obtained from Lemma~\ref{lem:median} under the joint loss
(\ref{eq:addloss}) (its explicit form is given in Theorem~\ref{thm:scale} below). Since
the joint adjustment through Lemma~\ref{lem:median} depends on the loss and generally
differs from the componentwise optimum, suppose
$\PR[T_{u,i}\ne T_{m,i}]=1$ for some component $i$. Then by the univariate optimality of
$T_{u,i}$ under its own component loss,
\[
\PR\big[h(T_{u,i},\theta_i)<h(T_{m,i},\theta_i)\big]\ge \tfrac12 \quad\text{for all }\bm\theta,
\]
and, by Lemma~\ref{lem:median} applied to the joint loss,
\[
\mathrm{PC}(T_m,T_u;\bm\theta)\ge \tfrac12 \quad\text{for all }\bm\theta .
\]
Now form $T_e$ by replacing one component of $T_m$ by the corresponding component of
$T_u$. Since each component problem is isomorphic, the one-dimensional comparison gives
$\mathrm{PC}(T_e,T_m;\bm\theta)\ge \mathrm{PC}(T_m,T_e;\bm\theta)$ for all $\bm\theta$,
while by Lemma~\ref{lem:median} $\mathrm{PC}(T_m,T_e;\bm\theta)\ge
\mathrm{PC}(T_e,T_m;\bm\theta)$; and $T_u$ in turn dominates $T_e$ in the modified
component. One thus obtains a cycle
\[
\mathrm{PC}(T_m,T_u)\ge\mathrm{PC}(T_u,T_m),\quad
\mathrm{PC}(T_u,T_e)\ge\mathrm{PC}(T_e,T_u),\quad
\mathrm{PC}(T_e,T_m)\ge\mathrm{PC}(T_m,T_e),
\]
with at least one strict domination, so the PCC is not transitive within the class of
equivariant estimators and no Pitman closest equivariant estimator exists in that class.
Moreover, the construction can be iterated: given any mixed estimator that uses the
univariate-optimal rule on a subset $A$ of the sample space of the maximal invariant and
the joint-optimal rule on its complement, one can choose a disjoint set $B$ with the same
probability and construct yet another estimator that is Pitman closer, so the cycle does
not terminate.
\end{example}

Example~\ref{ex:nontransitive} shows that although $T_e$ improves upon $T_m$ in one
component, applying different estimating rules to isomorphic component problems violates
the formal equivariance principle, which requires isomorphic problems to be solved by the
same rule (Berger, 1985). This motivates the following restriction.

\subsection{A permutation equivariance restriction}\label{sec:restriction}

A constructive way to formalize the requirement that isomorphic component problems be
solved by the same rule is to enlarge the transformation group by the coordinate
permutations. Let $S_p$ denote the symmetric group on $\{1,\dots,p\}$, acting on $p\times
n$ matrices by row permutation: for $\pi\in S_p$ with permutation matrix $P_\pi$, let
$g_\pi(X)=P_\pi X$.

\begin{lemma}\label{lem:formal}
Suppose $X=(\bm x_1,\dots,\bm x_n)$ is a $p\times n$ data matrix whose rows
$X_{i,\cdot}=(X_{i,1},\dots,X_{i,n})$ are independent with
$X_{i,\cdot}\sim \prod_{j=1}^n f(x_{ij}\mid\theta_i)$, and consider estimating
$\bm\theta=(\theta_1,\dots,\theta_p)$ under the additive loss (\ref{eq:addloss}) with a
common univariate loss $h$. Then the problem is invariant under the enlarged group
\begin{equation}\label{eq:enlarged}
G\times S_p=\big\{g_{C,\pi}(X)=CP_\pi X:\ C=\diag(c_1,\dots,c_p),\ c_i>0,\ \pi\in S_p\big\},
\end{equation}
with $\bar g_{C,\pi}(\bm\theta)=CP_\pi\bm\theta$ and $\tilde g_{C,\pi}(\bm
d)=C^rP_\pi\bm d$, and every estimator $\bm d$ equivariant under (\ref{eq:enlarged})
satisfies
\begin{equation}\label{eq:formal}
d_1(X_{i,\cdot})=\cdots=d_p(X_{i,\cdot}),
\end{equation}
i.e.\ the same estimating function is applied to each component sample.
\end{lemma}

\begin{proof}
Invariance of the model follows from the independence and identical distribution of the
rows, and invariance of the loss from
$L(P_\pi\bm d,P_\pi\bm\theta)=\sum_i h(d_{\pi(i)},\theta_{\pi(i)})=L(\bm d,\bm\theta)$.
For equivariance under a transposition $\pi=(i\,j)$, we have
$\bm d(P_\pi X)=P_\pi\bm d(X)$; evaluating component $i$ at a data matrix whose rows $i$
and $j$ agree gives $d_i=d_j$ as functions of a component sample, and varying over all
transpositions yields (\ref{eq:formal}).
\end{proof}

Three comments on this restriction are in order. First, the class defined by
(\ref{eq:formal}) is not new in itself: it is exactly the class of \emph{separable rules}
of Robbins' (1951) compound decision theory, which, together with its empirical Bayes
developments, studies precisely $p$ independent, structurally identical estimation tasks
under an additive loss and seeks optima within rules of the form
$\big(f(X_{1,\cdot}),\dots,f(X_{p,\cdot})\big)$ with a common $f$. What is new here is
the use of this restriction under the Pitman closeness criterion: in compound decision
theory separability is a convenience within which risk or empirical Bayes benchmarks are
defined, whereas here an analogue of the restriction is \emph{necessary} for a Pitman
closest equivariant estimator to exist at all, in view of
Example~\ref{ex:nontransitive}.

Second, the restriction has a clean interpretation in terms of the classical guideline
that the transformation group should match the parameter space (cf.\ Fraser, 1968; Eaton,
1989). The essential requirement behind that guideline is that the induced group $\bar G$
act transitively---and ideally freely---on $\Theta$, so that $\Theta$ is a single orbit
and equivariant comparisons become parameter-free. For the diagonal group $G$, $\bar
G\cong(\R_+)^p$ acts freely and transitively on $\Theta=(\R_+)^p$, and adjoining $S_p$
preserves transitivity while adding a genuine symmetry of the model, the loss, and the
estimand. The induced action of $(\R_+)^p\rtimes S_p$ is no longer free, but only on the
measure-zero diagonals $\{\theta_i=\theta_j\}$, and since $S_p$ is finite this does not
affect any Haar-measure considerations. The more accurate general principle, which the
present problem illustrates, is to use the largest group leaving the \emph{entire}
decision problem invariant whose induced action on $\Theta$ is transitive; the choice of
group is consequential, since it affects even the existence and admissibility of optimal
invariant rules (Kiefer, 1957; Zhou and Nayak, 2014). Note also that the enlargement
\emph{derives} the restricted class rather than merely justifying it: equivariance under
(\ref{eq:enlarged}), applied to estimators of the form $\bm w(Z)\odot\bm\delta^*(X)$ of
Theorem~\ref{thm:maxinv}, forces $w_1=\cdots=w_p$ with a common scalar $w(Z)$ invariant
under row permutations of $Z$.

Third, the restriction delineates its own boundary: it is available only when the
component problems are genuinely isomorphic, so that $S_p$ is a symmetry of the loss. In
the competing risks application of Section~5.2 the component families carry distinct
parameters $(\theta_1,\theta_2)$ entering the likelihood asymmetrically, the parameter
space is not transitive, and neither the enlargement nor the parameter-free reduction of
Theorem~\ref{thm:orbit} applies.

\begin{remark}\label{rem:class}
Even under (\ref{eq:formal}), the non-transitivity of Example~\ref{ex:nontransitive}
persists if the adjustment factor is allowed to vary with the maximal invariant in an
arbitrary way: estimators that switch between two admissible rules on disjoint subsets of
the invariant's range can produce non-terminating cycles. In view of this, and guided by
the role of completeness and sufficiency in the univariate theory (cf.\ Remark~2.3 of
Zhou and Nayak, 2012), in the sequel we restrict attention to equivariant estimators of
the form
\begin{equation}\label{eq:class}
\bm\delta(X)=w(Z)\,\bm\delta^*(X),
\end{equation}
where $\bm\delta^*$ is a fixed equivariant estimator, $Z$ is the maximal invariant, and
$w$ is a common scalar adjustment shared by all components. This class still permits
borrowing strength across components---rules $\delta_i(X)=f(X_{i,\cdot};X_{-i,\cdot})$
with symmetric dependence on the other rows satisfy (\ref{eq:formal})---while excluding
exactly the pathological mixing rules behind the cycles. Within it, a Pitman closest
member exists.
\end{remark}

\section{Multivariate scale family}\label{sec:scale}

\begin{definition}[Multivariate scale family]\label{def:scale}
Let $X=(\bm X_1,\dots,\bm X_n)$ be a $p\times n$ random matrix with joint density
\begin{equation}\label{eq:scalepdf}
f_P(\bm x_1,\dots,\bm x_n)=|P|^{-n} f\big(P^{-1}\bm x_1,\dots,P^{-1}\bm x_n\big),
\qquad \bm x_i\in\R^p,
\end{equation}
where $f$ is a given density and $P=\diag(\theta_1,\dots,\theta_p)$, $\theta_i>0$, is an
unknown diagonal scale matrix.
\end{definition}

Consider estimating $\bm\theta^{(r)}=(\theta_1^r,\dots,\theta_p^r)$ for a given
$r\in\R\setminus\{0\}$ under the loss
\[
L(\bm d,\bm\theta)=\sum_{i=1}^p h\!\left(\frac{d_i}{\theta_i^r}\right),
\]
where $h$ is strictly convex with $h(1)=0$; the scaled squared error loss
$h(t)=(t-1)^2$ is a leading special case. The problem is invariant under
$G=\{g_C(X)=CX\}$ with
\[
g_C(\bm\theta)=C\bm\theta,\qquad \tilde g_C(\bm d)=C^r\bm d,
\]
and an estimator $T(X)$ is equivariant if and only if
\begin{equation}\label{eq:scaleequiv}
T(CX)=C^rT(X)\quad\text{for all }C=\diag(c_1,\dots,c_p),\ c_i>0.
\end{equation}
The following representation result is the multivariate analogue of the univariate
characterization (Lehmann and Casella, 1998, p.~168); we give a proof because published
multivariate versions are incomplete.

\begin{theorem}\label{thm:maxinv}
Let $X$ have density (\ref{eq:scalepdf}) and suppose $\PR(X_{i,1}=0)=0$ for all $i$. Let
$\bm\delta^*(X)$ be an equivariant estimator of $\bm\theta^{(r)}$ with positive
components a.e., and define the $p\times n$ matrix $Z=(z_{i,j})$ by
\begin{equation}\label{eq:Z}
z_{i,j}=\frac{x_{i,j}}{x_{i,1}}\ (j\ne 1),\qquad
z_{i,1}=\frac{x_{i,1}}{|x_{i,1}|}.
\end{equation}
Then $Z$ is a maximal invariant under $G$, and $\bm\delta(X)$ is an equivariant
estimator of $\bm\theta^{(r)}$ if and only if there exists a vector function
$\bm w(Z)=(w_1(Z),\dots,w_p(Z))$ such that
\begin{equation}\label{eq:repr}
\bm\delta(X)=\bm w(Z)\odot\bm\delta^*(X),
\end{equation}
where $\odot$ denotes the Hadamard (componentwise) product.
\end{theorem}

\begin{proof}
Equivariance of $\bm\delta$ and $\bm\delta^*$ gives
$\delta_i(CX)=c_i^r\delta_i(X)$ and $\delta_i^*(CX)=c_i^r\delta_i^*(X)$, so the ratio
$w_i(X)=\delta_i(X)/\delta_i^*(X)$ is invariant: $w_i(CX)=w_i(X)$. Conversely every
function of the maximal invariant $Z$ is invariant. To see that an invariant $w_i$ must
depend on $X$ only through $Z$, take
$C^*=\diag(|x_{1,1}|^{-1},\dots,|x_{p,1}|^{-1})$; then $w_i(X)=w_i(C^*X)$, and the matrix
$C^*X$ is exactly determined by $Z$ in (\ref{eq:Z}). Hence $w_i(X)=w_i(Z)$, proving
(\ref{eq:repr}). That $Z$ is maximal invariant follows from the same construction: if
$Z(x)=Z(x')$ then $x'=Cx$ for $C=\diag(x'_{1,1}/x_{1,1},\dots)$.
\end{proof}

Under the formal equivariance restriction of Lemma~\ref{lem:formal} and the additive
common loss, $w_1(Z)=\cdots=w_p(Z)=w(Z)$, say, and we consider the class
(\ref{eq:class}). Since the parameter space is transitive under $\bar G$, by
Theorem~\ref{thm:orbit} the Pitman closeness of
$\bm\delta_a(X)=w_a(Z)\bm\delta^*(X)$ relative to $\bm\delta_b(X)=w_b(Z)\bm\delta^*(X)$
is independent of $\bm\theta$ and may be computed under $\bm\theta=\1$:
\begin{equation}\label{eq:pcscale}
\mathrm{PC}(\bm\delta_a,\bm\delta_b;\bm\theta)
=\PR_{\1}\!\big[L\big(w_a(Z)\bm\delta^*(X),\1\big)<L\big(w_b(Z)\bm\delta^*(X),\1\big)\big].
\end{equation}
Applying Lemma~\ref{lem:median} with $\theta=\1$ and $T=\bm\delta^*(X)$ yields the main
result of this section.

\begin{theorem}\label{thm:scale}
Let $X$ have density (\ref{eq:scalepdf}) and let $\bm T(X)=(T_1(X),\dots,T_p(X))$ be an
equivariant estimator of $\bm\theta^{(r)}$ with $T_i(X)>0$ a.e. Suppose the loss
$L(\bm d,\bm\theta)=\sum_{i=1}^p h(d_i/\theta_i^r)$ is strictly convex. Then, within the
class (\ref{eq:class}), a Pitman closest equivariant estimator of
$\bm\theta^{(r)}$ is
\begin{equation}\label{eq:pcescale}
\bm\delta_{\pi}(X)=m(Z)\,\bm T(X),
\end{equation}
where $m(Z)$ is a median of the conditional distribution, given $Z$ and
$\bm\theta=\1$, of the univariate random variable $c^*$ defined by
\begin{equation}\label{eq:cstar}
\bm T(X)^\intercal\,\nabla L\big(\1-c^*\bm T(X)\big)=0 .
\end{equation}
\end{theorem}

\begin{proof}
The loss is strictly convex as a sum of strictly convex functions. In
(\ref{eq:pcscale}), apply Lemma~\ref{lem:median} conditionally on $Z$ with $\theta=\1$
and $T=\bm T(X)$: for any competing adjustment $k(Z)$,
\[
\PR\big[L\big(m(Z)\bm T(X),\1\big)<L\big(k(Z)\bm T(X),\1\big)\big]
\ge \tfrac12\,\PR[k(Z)\ne m(Z)],
\]
which, by Theorem~\ref{thm:orbit}, holds for all $\bm\theta$.
\end{proof}

\begin{remark}\label{rem:basu}
As in the univariate case, one should take $\bm T(X)$ to be a function of a complete
sufficient statistic whenever one exists. Since the maximal invariant $Z$ is ancillary,
Basu's theorem then gives independence of $\bm T(X)$ and $Z$, and $m(Z)$ reduces to the
unconditional median of $c^*$ under $\bm\theta=\1$, which can be computed once and
tabulated.
\end{remark}

\begin{example}[Multivariate uniform distribution]\label{ex:uniform}
Let the columns of $X$ be i.i.d.\ from the $p$-dimensional rectangular uniform
distribution $U_p(\mathbf 0,\bm\theta)$, $\theta_i>0$, and consider estimating
$(\theta_1^r,\dots,\theta_p^r)$ under the scaled squared error loss
$L(\bm d,\bm\theta)=\sum_{i=1}^p\big[(d_i-\theta_i^r)/\theta_i^r\big]^2$. Let
$X_{i,(n)}=\max\{x_{i,1},\dots,x_{i,n}\}$. Then
$\bm T(X)=\big(X_{1,(n)}^r,\dots,X_{p,(n)}^r\big)$ is an equivariant estimator, and
(\ref{eq:cstar}) gives
\[
2\bm T(X)^\intercal\big(c^*\bm T(X)-\1\big)=0
\ \Longrightarrow\
c^*=\frac{\sum_{i=1}^p X_{i,(n)}^r}{\sum_{i=1}^p X_{i,(n)}^{2r}} .
\]
The Pitman closest equivariant estimator is
$\bm\delta_\pi(X)=m(Z)\,\bm T(X)$, and since
$(X_{1,(n)},\dots,X_{p,(n)})$ is complete and sufficient, by Remark~\ref{rem:basu} $m(Z)$
is the unconditional median of $c^*$ under $\bm\theta=\1$.
\end{example}

\begin{example}[Normal variances, zero means]\label{ex:normalscale}
Let the columns of $X$ be i.i.d.\ $N_p(\mathbf 0,\Sigma)$ with
$\Sigma=\diag(\sigma_1^2,\dots,\sigma_p^2)$, and consider estimating
$(\sigma_1^2,\dots,\sigma_p^2)$ under Anderson's quadratic loss
$L(D,\Sigma)=\tr\big(D\Sigma^{-1}-I\big)^2$, which for diagonal
$D=\diag(d_1,\dots,d_p)$ reduces to $\sum_{i=1}^p(d_i/\sigma_i^2-1)^2$. Let
$\bm T(X)=\big(\sum_{j=1}^n x_{1,j}^2,\dots,\sum_{j=1}^n x_{p,j}^2\big)$, an equivariant
estimator. Then (\ref{eq:cstar}) yields
\[
c^*=\frac{\sum_{i=1}^p\sum_{j=1}^n x_{i,j}^2}{\sum_{i=1}^p\big(\sum_{j=1}^n x_{i,j}^2\big)^2},
\]
and the Pitman closest equivariant estimator is $\bm\delta_\pi(X)=m(Z)\bm T(X)$. Since
$\bm T(X)$ is a function of the complete sufficient statistic, $m(Z)$ is the
unconditional median of $c^*$ under $\Sigma=I$.
\end{example}

\section{Multivariate location--scale family}\label{sec:locscl}

\begin{definition}[Multivariate location--scale family]\label{def:ls}
Let $X=(\bm X_1,\dots,\bm X_n)$ be a $p\times n$ random matrix with joint density
\begin{equation}\label{eq:lspdf}
f_{P,\bm\alpha}(\bm x_1,\dots,\bm x_n)
=|P|^{-n} f\big(P^{-1}(\bm x_1-\bm\alpha),\dots,P^{-1}(\bm x_n-\bm\alpha)\big),
\end{equation}
where $f$ is given, $\bm\alpha\in\R^p$ is an unknown location vector and
$P=\diag(\beta_1,\dots,\beta_p)$, $\beta_i>0$, an unknown diagonal scale matrix.
\end{definition}

The family (\ref{eq:lspdf}) is invariant under
$G=\{g_{\bm a,C}:\bm a\in\R^p,\ C=\diag(c_1,\dots,c_p),\ c_i>0\}$ with
\[
g_{\bm a,C}(X)=CX+\bm a\otimes\1^\intercal,\qquad
g_{\bm a,C}(\bm\alpha,\bm\beta)=(C\bm\alpha+\bm a,\ C\bm\beta).
\]

\subsection{Estimating powers of the scale parameters}

Consider estimating $\bm\beta^{(r)}=(\beta_1^r,\dots,\beta_p^r)$ under
$L(\bm d,\bm\beta)=\sum_{i=1}^p h(d_i/\beta_i^r)$ with strictly convex $h$, where
$\tilde g_{\bm a,C}(\bm d)=C^r\bm d$. An equivariant estimator satisfies
\begin{equation}\label{eq:lsequiv}
T(CX+\bm a\otimes\1^\intercal)=C^rT(X)
\end{equation}
and must in particular be location invariant. Let
$\bm U_i=\bm X_i-\bm X_n$, $i=1,\dots,n-1$; then $U=(\bm U_1,\dots,\bm U_{n-1})$ is a
maximal location invariant (Nayak, 1990), and every equivariant estimator of
$\bm\beta^{(r)}$ is a function of $U$. The density of $U$ has the form
\[
f_{P,\bm\alpha}(\bm u_1,\dots,\bm u_{n-1})
=|P|^{-(n-1)} f_1\big(P^{-1}\bm u_1,\dots,P^{-1}\bm u_{n-1}\big)
\]
for some $f_1$, so the problem reduces to the scale family of Section~\ref{sec:scale}.
Defining $Z$ from $U$ as in (\ref{eq:Z}) and applying Theorem~\ref{thm:scale} gives the
following.

\begin{theorem}\label{thm:lsscale}
Let $X$ have density (\ref{eq:lspdf}) and let $\bm T(U)=(T_1(U),\dots,T_p(U))$ be an
equivariant estimator of $\bm\beta^{(r)}$ based on $U$ with $T_i(U)>0$ a.e. Under the
strictly convex loss $L(\bm d,\bm\beta)=\sum_{i=1}^p h(d_i/\beta_i^r)$, a Pitman closest
equivariant estimator of $\bm\beta^{(r)}$ within the class (\ref{eq:class}) is
\[
\bm\delta_\pi(X)=m(Z)\,\bm T(U),
\]
where $m(Z)$ is a median of the conditional distribution given $Z$ and
$(\bm\alpha,\bm\beta)=(\mathbf 0,\1)$ of $c^*$ defined by
$\bm T(U)^\intercal\nabla L(\1-c^*\bm T(U))=0$.
\end{theorem}

\begin{example}[Normal variances, unknown means]\label{ex:normalvar}
Let the columns of $X$ be i.i.d.\ $N_p(\bm\mu,\Sigma)$ with
$\Sigma=\diag(\sigma_1^2,\dots,\sigma_p^2)$, and consider estimating
$(\sigma_1^2,\dots,\sigma_p^2)$ under Anderson's quadratic loss. With
$\bm U_i=\bm X_i-\bm X_n$ and
$\bm T(U)=\big(\sum_{j=1}^{n-1}u_{1,j}^2,\dots,\sum_{j=1}^{n-1}u_{p,j}^2\big)$,
equation (\ref{eq:cstar}) gives
\[
c^*=\frac{\sum_{i=1}^p\sum_{j=1}^{n-1}u_{i,j}^2}
{\sum_{i=1}^p\big(\sum_{j=1}^{n-1}u_{i,j}^2\big)^2},
\]
and the Pitman closest equivariant estimator is $\bm\delta_\pi(X)=m(Z)\bm T(U)$. Since
$\bm T(U)$ is a function of the complete sufficient statistic, $m(Z)$ is the
unconditional median of $c^*$ under $\Sigma=I$.
\end{example}

\subsection{Estimating the location parameters}

Now consider estimating $\bm\alpha$ under
$L(\bm d,\bm\alpha,\bm\beta)=\sum_{i=1}^p h\big((d_i-\alpha_i)/\beta_i\big)$ with
strictly convex $h$; then $\tilde g_{\bm a,C}(\bm d)=C\bm d+\bm a$ and equivariance means
\begin{equation}\label{eq:locequiv}
\bm\delta(CX+\bm a\otimes\1^\intercal)=C\bm\delta(X)+\bm a .
\end{equation}
For any two equivariant estimators, their difference is location invariant and scale
equivariant, so by Theorem~\ref{thm:maxinv} and the discussion above, given a reference
equivariant estimator $\bm T(X)$ and a componentwise positive, location invariant and
scale equivariant statistic $\bm T_e(U)$, every equivariant estimator in the restricted
class has the form
\begin{equation}\label{eq:locclass}
\bm\delta(X)=\bm T(X)-w(Z)\,\bm T_e(U)
\end{equation}
with a common scalar $w(Z)$. Transitivity of the parameter space and
Theorem~\ref{thm:orbit} reduce all comparisons to $(\bm\alpha,\bm\beta)=(\mathbf 0,\1)$,
and Lemma~\ref{lem:median} yields the following.

\begin{theorem}\label{thm:loc}
Let $\bm T(X)$ be an equivariant estimator of $\bm\alpha$ and $\bm T_e(U)$ a
location invariant, scale equivariant statistic with positive components a.e. Under the
strictly convex loss
$L(\bm d,\bm\alpha,\bm\beta)=\sum_{i=1}^p h\big((d_i-\alpha_i)/\beta_i\big)$, a Pitman
closest equivariant estimator of $\bm\alpha$ within the class (\ref{eq:locclass}) is
\begin{equation}\label{eq:pceloc}
\bm\delta_\pi(X)=\bm T(X)-m(Z)\,\bm T_e(U),
\end{equation}
where $m(Z)$ is a median of the conditional distribution given $Z$ and
$(\bm\alpha,\bm\beta)=(\mathbf 0,\1)$ of $c^*$ defined by
\[
\bm T_e(U)^\intercal\,\nabla L\big(\bm T(X)-c^*\bm T_e(U)\big)=0 .
\]
\end{theorem}

\begin{remark}
Choosing $(\bm T,\bm T_e)$ as functions of a complete sufficient statistic makes them
independent of the ancillary $Z$, so $m(Z)$ is the unconditional median of $c^*$ under
$(\bm\alpha,\bm\beta)=(\mathbf 0,\1)$.
\end{remark}

\begin{example}[Normal means]\label{ex:normalmean}
In the setup of Example~\ref{ex:normalvar}, consider estimating
$\bm\mu=(\mu_1,\dots,\mu_p)$ under the scaled squared error loss
$L(\bm d,\bm\mu)=\sum_{i=1}^p\big[(d_i-\mu_i)/\sigma_i\big]^2$. Take
\[
\bm T(X)=\Big(\tfrac1n\textstyle\sum_j x_{1,j},\dots,\tfrac1n\textstyle\sum_j x_{p,j}\Big),
\qquad
\bm T_e(U)=\Big(\sqrt{\textstyle\sum_j u_{1,j}^2},\dots,\sqrt{\textstyle\sum_j u_{p,j}^2}\Big).
\]
Then $\bm T$ is location--scale equivariant, $\bm T_e$ is location invariant and scale
equivariant, and
\[
c^*=\frac{\sum_{i=1}^p \bar x_i\sqrt{\sum_j u_{i,j}^2}}{\sum_{i=1}^p\sum_j u_{i,j}^2},
\qquad \bar x_i=\tfrac1n\textstyle\sum_j x_{i,j}.
\]
The Pitman closest equivariant estimator of $\bm\mu$ is
$\bm\delta_\pi(X)=\bm T(X)-m(Z)\bm T_e(U)$, where, by completeness and sufficiency of
$(\bm T,\bm T_e)$, $m(Z)$ is the unconditional median of $c^*$ under
$(\bm\mu,\Sigma)=(\mathbf 0,I)$.
\end{example}

\begin{example}[Multivariate uniform location]\label{ex:uniformloc}
Let the columns of $X$ be i.i.d.\ $U_p(\bm\alpha,\bm\beta)$ (componentwise uniform on
$(\alpha_i,\beta_i)$), and consider estimating $\bm\alpha$ under
$L(\bm d,\bm\alpha)=\sum_{i=1}^p\big[(d_i-\alpha_i)/(\beta_i-\alpha_i)\big]^2$. With
$X_{i,(1)}=\min_j x_{i,j}$ and $X_{i,(n)}=\max_j x_{i,j}$, take
$\bm T(X)=(X_{1,(1)},\dots,X_{p,(1)})$ and
$\bm T_e(U)=(X_{1,(n)}-X_{1,(1)},\dots,X_{p,(n)}-X_{p,(1)})$. Then
\[
c^*=\frac{\sum_{i=1}^p X_{i,(1)}\big(X_{i,(n)}-X_{i,(1)}\big)}
{\sum_{i=1}^p\big(X_{i,(n)}-X_{i,(1)}\big)^2},
\]
and the Pitman closest equivariant estimator of $\bm\alpha$ is
$\bm\delta_\pi(X)=\bm T(X)-m(Z)\bm T_e(U)$, with $m(Z)$ the unconditional median of
$c^*$ under $\bm\alpha=\mathbf 0$, $\bm\beta=\1$. The scale parameters $\bm\beta$ are
treated similarly via Theorem~\ref{thm:lsscale}.
\end{example}

\section{Applications}\label{sec:app}

Location--scale families are widely used to model lifetimes in reliability and survival
analysis, and the Pitman closeness comparison of estimators in specific lifetime
distributions has attracted continued interest (e.g.\ Davies, 2021; Volovskiy and Kamps,
2021). In this section we evaluate the proposed estimators for the Rayleigh distribution
and for a competing risks model with Rayleigh component lifetimes, comparing them with
the maximum likelihood estimator (MLE) and Bayes estimators both under the Pitman
closeness criterion and under average-error criteria.

\subsection{The Rayleigh distribution}

The Rayleigh distribution with scale parameter $\sigma$ has density
$f(x;\sigma)=(2x/\sigma^2)e^{-x^2/\sigma^2}$, $x>0$. Writing $\theta=\sigma^2$,
\begin{equation}\label{eq:rayleigh}
f(x;\theta)=\frac{2x}{\theta}\,e^{-x^2/\theta},\qquad x>0,\ \theta>0,
\end{equation}
which is a scale family in $\theta$. For a random sample $X=(X_1,\dots,X_n)$, the MLE is
\begin{equation}\label{eq:mle}
\hat\theta_{\mathrm{MLE}}=\frac1n\sum_{i=1}^n x_i^2 .
\end{equation}
With the conjugate inverse gamma prior $\theta\sim I\Gamma(\alpha,\beta)$, the posterior
is $I\Gamma\big(n+\alpha,\sum_i x_i^2+\beta\big)$, and minimizing the posterior risk
under the scaled squared error loss gives
\begin{equation}\label{eq:bayes}
\hat\theta_{\mathrm{Bayes}}=\frac{\sum_{i=1}^n x_i^2+\beta}{\alpha+n+1}.
\end{equation}
The MLE (\ref{eq:mle}) is equivariant and is a complete sufficient statistic. Since
$X_i^2\sim\mathrm{Exp}(\theta)$, under $\theta=1$ we have $\frac1n\sum_i X_i^2\sim
\Gamma(n,1/n)$, and by Theorem~\ref{thm:scale} (with $p=1$) the Pitman closest
equivariant estimator is
\begin{equation}\label{eq:pce}
\hat\theta_{\pi}=\frac{\frac1n\sum_{i=1}^n x_i^2}{m_n},
\end{equation}
where $m_n$ is the median of the $\Gamma(n,1/n)$ distribution, computed numerically.

\subsubsection{Simulation study}

For sample sizes $n=10,20,30,100$ and true values $\theta=0.5,1,2$, we generated 10{,}000
samples from (\ref{eq:rayleigh}) and computed (\ref{eq:mle}), (\ref{eq:pce}) and
(\ref{eq:bayes}), the latter with (i) prior mean equal to the true $\theta$ with
$\alpha=1.1$ and $\alpha=2$, and (ii) an iteratively updated prior starting from
$\alpha=\beta=1.5$. Figure~\ref{fig:gpc} displays the pairwise generalized Pitman
closeness under the scaled squared error loss for $\theta=0.5$; the picture under the
absolute error loss is identical up to Monte Carlo error, confirming that the Pitman
closeness comparison is the same for all strictly convex losses, so the estimator
(\ref{eq:pce}), derived under the scaled squared error loss, remains optimal under every
strictly convex loss. In all configurations $\hat\theta_\pi$ is Pitman closer than both
the MLE and each Bayes estimator.

\begin{figure}[htbp]
\centering
\includegraphics[width=\textwidth]{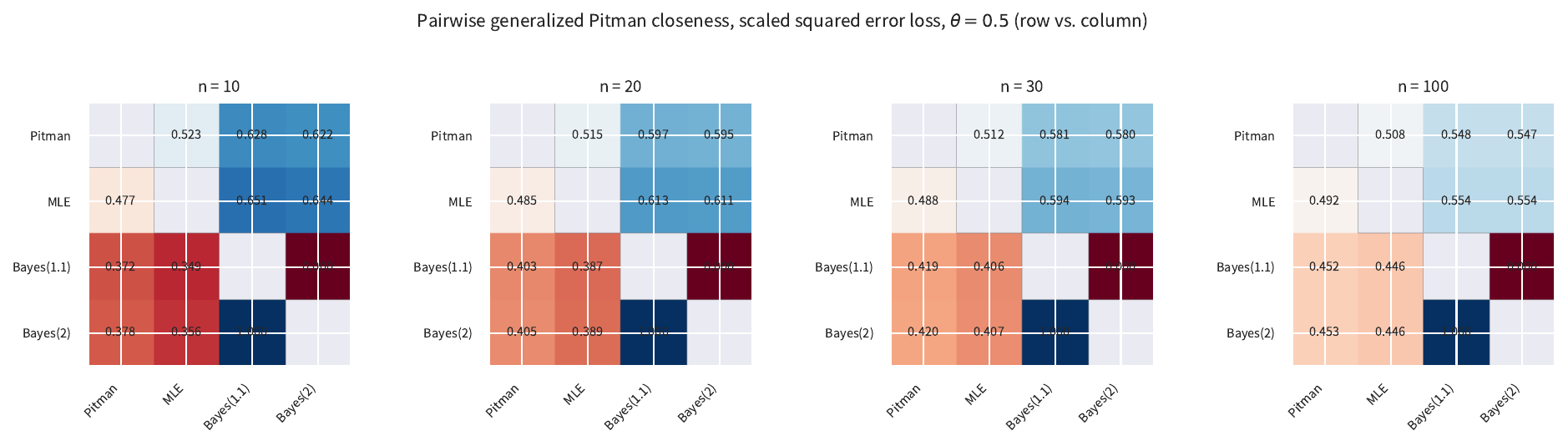}
\caption{Pairwise generalized Pitman closeness among the Pitman closest equivariant
estimator, the MLE and two Bayes estimators for the Rayleigh distribution, scaled
squared error loss, $\theta=0.5$, 100{,}000 replications. Entry $(i,j)$ is
$\PR[L(T_i,\theta)<L(T_j,\theta)]$.}
\label{fig:gpc}
\end{figure}

Tables~\ref{tab:msse} and \ref{tab:mae} report the mean scaled squared error and the mean
absolute error. Under average-error criteria no estimator dominates uniformly: a Bayes
estimator with a well-tuned prior can have the smallest average error, but its
performance depends heavily on the prior hyperparameters, and for large samples all
methods are close. Under the Pitman closeness criterion, however, $\hat\theta_\pi$ is
uniformly the best, and it is completely free of prior inputs.

\begin{table}[htbp]
\centering\small
\caption{Mean scaled squared error of the estimators, Rayleigh distribution.}
\label{tab:msse}
\begin{tabular}{cc ccccc}
\toprule
$\theta$ & $n$ & Pitman & MLE & Bayes$_{1.1}$ & Bayes$_{2}$ & Bayes$_{\mathrm{iter}}$\\
\midrule
\multirow{4}{*}{0.5}
& 10 & 0.1111 & 0.1023 & 0.0956 & 0.0828 & 0.0971\\
& 20 & 0.0523 & 0.0503 & 0.0491 & 0.0454 & 0.0496\\
& 30 & 0.0332 & 0.0325 & 0.0324 & 0.0307 & 0.0326\\
& 100 & 0.0099 & 0.0099 & 0.0099 & 0.0097 & 0.0099\\
\midrule
\multirow{4}{*}{1.0}
& 10 & 0.1106 & 0.1020 & 0.0958 & 0.0830 & 0.0973\\
& 20 & 0.0522 & 0.0502 & 0.0496 & 0.0458 & 0.0501\\
& 30 & 0.0334 & 0.0336 & 0.0335 & 0.0317 & 0.0337\\
& 100 & 0.0099 & 0.0099 & 0.0098 & 0.0097 & 0.0099\\
\midrule
\multirow{4}{*}{2.0}
& 10 & 0.1084 & 0.1000 & 0.0947 & 0.0821 & 0.0963\\
& 20 & 0.0521 & 0.0501 & 0.0492 & 0.0454 & 0.0496\\
& 30 & 0.0337 & 0.0330 & 0.0329 & 0.0311 & 0.0331\\
& 100 & 0.0098 & 0.0097 & 0.0097 & 0.0095 & 0.0097\\
\bottomrule
\end{tabular}
\end{table}

\begin{table}[htbp]
\centering\small
\caption{Mean absolute error of the estimators, Rayleigh distribution.}
\label{tab:mae}
\begin{tabular}{cc ccccc}
\toprule
$\theta$ & $n$ & Pitman & MLE & Bayes$_{1.1}$ & Bayes$_{2}$ & Bayes$_{\mathrm{iter}}$\\
\midrule
\multirow{4}{*}{0.5}
& 10 & 0.1289 & 0.1259 & 0.1291 & 0.1202 & 0.1301\\
& 20 & 0.0910 & 0.0897 & 0.0909 & 0.0873 & 0.0913\\
& 30 & 0.0723 & 0.0717 & 0.0731 & 0.0711 & 0.0733\\
& 100 & 0.0396 & 0.0395 & 0.0399 & 0.0395 & 0.0399\\
\midrule
\multirow{4}{*}{1.0}
& 10 & 0.2592 & 0.2516 & 0.2575 & 0.2396 & 0.2595\\
& 20 & 0.1814 & 0.1789 & 0.1824 & 0.1752 & 0.1832\\
& 30 & 0.1476 & 0.1463 & 0.1491 & 0.1450 & 0.1495\\
& 100 & 0.0791 & 0.0789 & 0.0791 & 0.0784 & 0.0792\\
\midrule
\multirow{4}{*}{2.0}
& 10 & 0.5105 & 0.4960 & 0.5142 & 0.4786 & 0.5185\\
& 20 & 0.3607 & 0.3556 & 0.3625 & 0.3483 & 0.3641\\
& 30 & 0.2898 & 0.2872 & 0.2940 & 0.2859 & 0.2949\\
& 100 & 0.1575 & 0.1571 & 0.1579 & 0.1566 & 0.1581\\
\bottomrule
\end{tabular}
\end{table}

\subsubsection{Real data analysis: ball bearing failures}

We applied the three estimators to the classical deep-groove ball bearing data set
(Lawless, 2011), consisting of the numbers of millions of revolutions before failure for
23 bearings. The Kolmogorov--Smirnov test (statistic 0.1375, $p$-value 0.7273) and a
chi-squared goodness-of-fit test (statistic 1.7790, $p$-value 0.7763), together with the
Q--Q plot, support the Rayleigh model. To assess robustness, we removed two observations
at random ten times, refitted the model by each method (Bayes with $\alpha=2$, $\beta=2$),
and computed the K--S statistic of the fitted model against the full data set; the
results are in Table~\ref{tab:ks}. The Bayes estimator is occasionally the best and
occasionally the worst, while the MLE and the Pitman closest equivariant estimator are
stable and nearly indistinguishable, consistent with the simulation study.

\begin{table}[htbp]
\centering\small
\caption{K--S statistics of models fitted with two deleted observations (indices shown),
evaluated against the full ball bearing data set.}
\label{tab:ks}
\begin{tabular}{cccc}
\toprule
Deleted data & MLE & Pitman & Bayes$_{1.1}$\\
\midrule
(09,\,21) & 0.1212 & 0.1259 & 0.1393\\
(03,\,15) & 0.1483 & 0.1544 & 0.1221\\
(14,\,19) & 0.1301 & 0.1363 & 0.1335\\
(07,\,19) & 0.1355 & 0.1406 & 0.1301\\
(08,\,09) & 0.1559 & 0.1618 & 0.1173\\
(02,\,11) & 0.1557 & 0.1616 & 0.1174\\
(11,\,12) & 0.1470 & 0.1519 & 0.1229\\
(14,\,22) & 0.1374 & 0.1337 & 0.1693\\
(11,\,22) & 0.1371 & 0.1333 & 0.1689\\
(15,\,20) & 0.1168 & 0.1162 & 0.1467\\
\bottomrule
\end{tabular}
\end{table}

\subsection{Competing risks model with Rayleigh component lifetimes}

Competing risks data arise when a unit can fail from any of $K$ mutually exclusive causes
and only the minimum lifetime and its cause are observed (Cox, 1959). Let the latent
lifetime of unit $i$ under cause $k$ be $X_{i,k}$, independent across $k$, with
$X_{i,k}\sim\mathrm{Rayleigh}(\theta_k)$ in the parametrization (\ref{eq:rayleigh}). The
observed data are $(X_i,\delta_i)$, $i=1,\dots,n$, where $X_i=\min\{X_{i,1},\dots,X_{i,K}\}$
and $\delta_i\in\{1,\dots,K\}$ is the failure cause. Estimating
$(\theta_1,\dots,\theta_K)$ from such data is a parameter estimation problem for a
multivariate scale family, to which the theory of Section~\ref{sec:scale} applies. We
focus on $K=2$.

Writing $n_k=\sum_{i=1}^n I(\delta_i=k)$, the likelihood factorizes and the MLEs are
\begin{equation}\label{eq:crmle}
\hat\theta_{k,\mathrm{MLE}}=\frac{\sum_{i=1}^n x_i^2}{n_k},\qquad k=1,2.
\end{equation}
Under independent inverse gamma priors $\theta_k\sim I\Gamma(\alpha_k,\beta_k)$, the
posteriors are independent
$I\Gamma\big(\alpha_k+n_k,\ \beta_k+\sum_i x_i^2\big)$, and the Bayes estimators under the
scaled squared error loss are
\begin{equation}\label{eq:crbayes}
\hat\theta_{k,\mathrm{Bayes}}=\frac{\beta_k+\sum_{i=1}^n x_i^2}{\alpha_k+n_k+1}.
\end{equation}
The MLEs (\ref{eq:crmle}) are equivariant, and
$\big(\sum_i x_i^2/n_1,\ \sum_i x_i^2/n_2\big)$ is a function of the complete sufficient
statistic for $(\theta_1,\theta_2)$. The Pitman closest equivariant estimator given by
Lemma~\ref{lem:median} and Theorem~\ref{thm:scale} is
\begin{equation}\label{eq:crpce}
\hat\theta_{k,\pi}=m(c^*)\,\frac{\sum_{i=1}^n x_i^2}{n_k},\qquad k=1,2,
\end{equation}
where
\begin{equation}\label{eq:cstarcr}
c^*=\left(\frac{\sum_i x_i^2}{\theta_1 n_1}+\frac{\sum_i x_i^2}{\theta_2 n_2}\right)
\Bigg/
\left[\left(\frac{\sum_i x_i^2}{\theta_1 n_1}\right)^2
+\left(\frac{\sum_i x_i^2}{\theta_2 n_2}\right)^2\right]
\end{equation}
and $m(c^*)$ is the median of $c^*$. Since the parameter space here is not transitive
($n_1,n_2$ depend on the parameters), Theorem~\ref{thm:orbit} does not reduce the
comparison to a single parameter value; in the simulation we evaluate $m(c^*)$ by Monte
Carlo with $(\theta_1,\theta_2)$ replaced by the MLEs.

For $n=10,20,30,100$ and $(\theta_1,\theta_2)=(0.5,2)$ and $(1,2)$, we generated 10{,}000
competing risks samples. In every configuration the Pitman closest equivariant estimator
is Pitman closer than the MLE and the Bayes estimator (with prior means set at the true
values, $\alpha_1=\alpha_2=2$), in agreement with the theory. Table~\ref{tab:cr} reports
the mean errors: the multivariate Pitman closest equivariant estimator also has smaller
mean scaled squared error and mean absolute error than the MLE at all sample sizes. We do
not have a theoretical proof of this average-error superiority; it is an empirical
finding. The Bayes estimator can have still smaller average error, but only with prior
hyperparameters tuned using the true parameter values.

\begin{table}[htbp]
\centering\small
\caption{Mean scaled squared error and mean absolute error, competing risks model with
Rayleigh component lifetimes.}
\label{tab:cr}
\begin{tabular}{cc ccc ccc}
\toprule
& & \multicolumn{3}{c}{Mean scaled squared error} & \multicolumn{3}{c}{Mean absolute error}\\
\cmidrule(lr){3-5}\cmidrule(lr){6-8}
$(\theta_1,\theta_2)$ & $n$ & Pitman & MLE & Bayes & Pitman & MLE & Bayes\\
\midrule
\multirow{4}{*}{(0.5,\,2.0)}
& 10 & 0.5119 & 0.6682 & 0.2999 & 1.1368 & 1.2247 & 0.9750\\
& 20 & 0.7689 & 0.9164 & 0.1717 & 1.1227 & 1.1927 & 0.6986\\
& 30 & 0.5670 & 0.6412 & 0.1353 & 0.9129 & 0.9523 & 0.6107\\
& 100 & 0.0755 & 0.0789 & 0.0542 & 0.4172 & 0.4230 & 0.3739\\
\midrule
\multirow{4}{*}{(1.0,\,2.0)}
& 10 & 0.7960 & 1.0796 & 0.2706 & 1.3411 & 1.4677 & 0.9763\\
& 20 & 0.4089 & 0.4852 & 0.1642 & 0.9361 & 0.9863 & 0.7383\\
& 30 & 0.2033 & 0.2286 & 0.1200 & 0.7205 & 0.7425 & 0.6259\\
& 100 & 0.0483 & 0.0503 & 0.0423 & 0.3819 & 0.3863 & 0.3643\\
\bottomrule
\end{tabular}
\end{table}

\section{Discussion}

For multivariate scale and location--scale families with independent components, we have
extended the univariate theory of Pitman closest equivariant estimation of Nayak (1990)
and Zhou and Nayak (2012). Two features of the multivariate problem deserve emphasis.
First, the non-transitivity of the Pitman closeness criterion, which equivariance
neutralizes in the univariate case, resurfaces in higher dimensions, and an additional,
weak restriction—motivated by the formal equivariance principle—is needed for an optimum
to exist. This mirrors, in the Pitman closeness context, the classical inadmissibility
phenomena for best invariant rules under risk criteria (Stein, 1956; Portnoy and Stein,
1971), where extra restrictions are likewise needed. Second, once the class is suitably
restricted, the optimum estimators retain the attractive features of the univariate
theory: they are median adjustments of any given equivariant estimator, they depend on
the loss only through strict convexity, and the adjustment factors can be tabulated in
advance when a complete sufficient statistic exists.

Several directions remain open. Our main results assume independent components; for
dependent components, e.g.\ a general multivariate normal with unrestricted covariance
matrix, the problem belongs to the location--scatter family, where even the
characterization of equivariant estimators is delicate (Anderson, 2003, p.~280; Maronna,
1976) and the choice of transformation group affects the existence and form of optima.
The competing risks application suggests, but does not prove, average-error superiority
of the Pitman closest estimator over the MLE; a theoretical explanation would be
valuable. Finally, the prediction framework of Zhou and Nayak (2012) and estimation from
censored samples are natural settings for multivariate extensions of the present results.

\section*{Author contributions}
The theoretical results, simulation experiments, and data analysis were developed in the
Master's thesis of Y.L.\ (Sun Yat-sen University, 2022) under the supervision of H.Z.\
(conceptualization and methodology). Both authors contributed to the interpretation of
the results and approved the final manuscript.

\section*{Acknowledgement}
The authors thank the School of Mathematics, Sun Yat-sen University, for its support. The
English manuscript was prepared from the thesis with the assistance of an AI tool (Kimi,
Moonshot AI), which restructured and translated the material, reformulated the
restriction of Section~\ref{sec:restriction} in group-theoretic terms, and independently
reran the Monte Carlo verification of the simulation results; the authors reviewed all
content and take full responsibility for it.


\begin{thebibliography}{99}\small

\bibitem{anderson} Anderson, T.W., 2003. An Introduction to Multivariate Statistical
Analysis, third ed. Wiley, New York.

\bibitem{berger} Berger, J.O., 1985. Statistical Decision Theory and Bayesian Analysis,
second ed. Springer, New York.

\bibitem{boyd} Boyd, S., Vandenberghe, L., 2004. Convex Optimization. Cambridge
University Press, Cambridge.

\bibitem{cox} Cox, D.R., 1959. The analysis of exponentially distributed life-times with
two types of failure. Journal of the Royal Statistical Society: Series B 21, 411--421.

\bibitem{davies} Davies, K.F., 2021. Pitman closeness results for Type-I hybrid censored
data from exponential distribution. Journal of Statistical Computation and Simulation 91,
58--80.

\bibitem{evans} Evans, L.C., Gariepy, R.F., 1992. Measure Theory and Fine Properties of
Functions. CRC Press, Boca Raton.

\bibitem{eaton} Eaton, M.L., 1989. Group Invariance Applications in Statistics. Institute
of Mathematical Statistics, Hayward, CA.

\bibitem{fraser} Fraser, D.A.S., 1968. The Structure of Inference. Wiley, New York.

\bibitem{ghosh} Ghosh, M., Sen, P.K., 1989. Median unbiasedness and Pitman closeness.
Journal of the American Statistical Association 84, 1089--1091.

\bibitem{keating} Keating, J.P., Mason, R.L., Sen, P.K., 1993. Pitman's Measure of
Closeness: A Comparison of Statistical Estimators. SIAM, Philadelphia.

\bibitem{kiefer} Kiefer, J., 1957. Invariance, minimax sequential estimation, and
continuous time processes. The Annals of Mathematical Statistics 28, 573--601.

\bibitem{kourouklis1} Kourouklis, S., 1995a. Estimating powers of the scale parameter of
an exponential distribution with unknown location under Pitman's measure of closeness.
Journal of Statistical Planning and Inference 48, 185--195.

\bibitem{kourouklis2} Kourouklis, S., 1995b. Estimation of an exponential quantile under
Pitman's measure of closeness. Canadian Journal of Statistics 23, 257--268.

\bibitem{kourouklis3} Kourouklis, S., 1996. Improved estimation under Pitman's measure of
closeness. Annals of the Institute of Statistical Mathematics 48, 509--518.

\bibitem{kubokawa1} Kubokawa, T., 1990. Estimating powers of the generalized variance
under the Pitman closeness criterion. Canadian Journal of Statistics 18, 59--62.

\bibitem{kubokawa2} Kubokawa, T., 1991. Equivariant estimation under the Pitman closeness
criterion. Communications in Statistics---Theory and Methods 20, 3499--3523.

\bibitem{lawless} Lawless, J.F., 2011. Statistical Models and Methods for Lifetime Data.
John Wiley \& Sons, Hoboken.

\bibitem{lehmann} Lehmann, E.L., Casella, G., 1998. Theory of Point Estimation, second
ed. Springer, New York.

\bibitem{maronna} Maronna, R., 1976. Robust $M$-estimators of multivariate location and
scatter. The Annals of Statistics 4, 51--67.

\bibitem{nayak} Nayak, T.K., 1990. Estimation of location and scale parameters using
generalized Pitman nearness criterion. Journal of Statistical Planning and Inference 24,
259--268.

\bibitem{nayak98} Nayak, T.K., 1998. On equivariant estimation of the location of
elliptical distributions under Pitman closeness criterion. Statistics \& Probability
Letters 36, 373--378.

\bibitem{peddada} Peddada, S.D., 1985. A short note on Pitman's measure of nearness. The
American Statistician 39, 298--299.

\bibitem{pitman} Pitman, E.J.G., 1937. The closest estimates of statistical parameters.
Proceedings of the Cambridge Philosophical Society 33, 212--222.

\bibitem{portnoy} Portnoy, S., Stein, C., 1971. Inadmissibility of the best invariant
test in three or more dimensions. The Annals of Mathematical Statistics 42, 799--801.

\bibitem{rao} Rao, C.R., Keating, J.P., Mason, R.L., 1986. The Pitman nearness criterion
and its determination. Communications in Statistics---Theory and Methods 15, 3173--3191.

\bibitem{robbins} Robbins, H., 1951. Asymptotically subminimax solutions of compound
statistical decision problems. In: Proceedings of the Second Berkeley Symposium on
Mathematical Statistics and Probability. University of California Press, Berkeley,
pp.~131--148.

\bibitem{stein} Stein, C., 1956. Inadmissibility of the usual estimator for the mean of a
multivariate normal distribution. In: Proceedings of the Third Berkeley Symposium on
Mathematical Statistics and Probability, vol.~1. University of California Press,
Berkeley, pp.~197--206.

\bibitem{volovskiy} Volovskiy, G., Kamps, U., 2021. Comparison of likelihood-based
predictors of future Pareto and Lomax record values in terms of Pitman closeness.
Communications in Statistics---Theory and Methods, in press.

\bibitem{zhou2012} Zhou, H., Nayak, T.K., 2012. Pitman closest equivariant estimators and
predictors under location--scale models. Journal of Statistical Planning and Inference
142, 1367--1377.

\bibitem{zhou2014} Zhou, H., Nayak, T.K., 2014. A note on existence and construction of
invariant loss functions. Statistics 48, 1335--1343.

\end{thebibliography}
\end{document}